\documentclass[12pt,reqno]{amsart}

\usepackage[mathlines]{lineno}

\usepackage{times}
\usepackage{amsfonts}
\usepackage{amssymb}
\usepackage{latexsym}
\usepackage{xcolor}

\newtheorem{theorem}{Theorem}
\newtheorem{lemma}[theorem]{Lemma}

\newtheorem{definition}[theorem]{Definition}
\newtheorem{remark}{Remark}
\newtheorem{corollary}[theorem]{Corollary}
\newtheorem{example}{Example}

\begin{document}

\begin{center}
\Large{On the spectra of  multidimensional normal discrete Hausdorff operators}
\end{center}

\centerline{A. R. Mirotin}

\centerline{amirotin@yandex.ru}

\

\textsc{Abstract.} In the paper the general case of a normal discrete Hausdorff operators in $L^2(\mathbb{R}^d)$  is considered. The main result states that under some natural arithmetic condition the spectrum of such an operator is rotationally invariant. 
Several special cases and examples  are considered. 
\


\


\
\

\section{ Introduction}

In resent two  decades different notions of a Hausdorff operator have been suggested (see e.g., \cite{Ls, CFW, KL, JMAA} and bibliography therein).

Hausdorff operators over the topological group $\mathbb{R}^d$ in a sense of \cite{LL}  (see also \cite{JMAA}) have attracted much attention.
But most of the work in this direction was devoted to the boundedness of such operators in various function spaces , see \cite{CDFZ, CFL, LL, LM, RF} among others. Probably the only exception is a special case of Hausdorff operators, namely the Ces\`{a}ro operator (e.g., \cite{Albanese-Bonet-Ricker, Curbera_Ricker}).

The situation has changed somewhat in recent years. It is shown in \cite{faa, Forum} that every normal Hausdorff operator in
 $L^2(\mathbb{R}^d)$ with self-adjoint perturbation matrices is unitarily equivalent to the multiplication operator
by some matrix-valued function (its matrix symbol) in the space $L^2(\mathbb{R}^d; \mathbb{C}^{2^d})$. This
is an analogue of the Spectral Theorem for the class of operators under consideration. This allowed  to find the norm, to study  the spectrum,
 and to develop functional calculi of such operators \cite{LifMir}.  See also \cite{RJMP}, where the case of  the spaces $L^p(\mathbb{R}^d)$ was considered.

The structure of the spectrum of discrete normal Hausdorff operators in $L^2(\mathbb{R}^d)$  was investigated in \cite{JMS}  for the case of positive or negative definite perturbation matrices. In this paper, the general case of normal discrete Hausdorff operators in $L^2(\mathbb{R}^d)$  is considered. Our interest to such operators is motivated by the fact that discrete Hausdorff operators 
are involved in  functional differential equations (see, e.g., \cite{BellenZennaro, Liu-Li, Ross, Ross2, Zaidi-Van Brunt-Wake} and the bibliography therein, see also Section 4 below), problems in analysis    (see, e.g.,  \cite{makarov}, especially Theorem 1 and formula (12) therein), and quantum mechanics \cite[Chapter VII, \S 2, (7.24)]{MF}.

The main result of this article states that under some natural  arithmetic condition the spectrum of normal discrete Hausdorff  operator is rotationally invariant. 
Several special cases  and examples of functional differential equations  are considered.

\section{Preliminaries}\setcounter{equation}{0}
We study the following special case of a  Hausdorff operator on Euclidean spaces.

\begin{definition}\label{discr}
Let $c=(c(k))_{k\in \mathbb{Z}}$ be a sequences of complex numbers, and $A(k)\in \mathrm{GL}(d,\mathbb{R})$ for all $k\in \mathbb{Z}$. \textit{ A discrete Hausdorff operator} acts on a function $f:\mathbb{R}^d\to \mathbb{C}$
by the rule
\[
\mathcal{H}_{c,A}f(x)=\sum_{k=-\infty}^\infty c(k) f(A(k) x)
\]
 provided the series converges absolutely.
\end{definition}

Without loss of generality one can assume that $A(k)\ne A(l)$ for $k\ne l$.

At the first time such operators in this general form appeared in \cite[Section 4]{CFL} (see also \cite{faa, Forum}), but one can find a lot of  concrete examples in the literature,
see, e.~g., \cite{makarov}, \cite{Ross2}, \cite{Ross} and the bibliography therein, and examples of functional-differential equations below.

\begin{lemma}\label{lem1}
(i)  Let $1\le p\le \infty$. If
$$
N_p(c,A):=\sum_{k=-\infty}^\infty |c(k)||\det A(k)|^{-1/p}
<\infty,
$$
 then the operator $\mathcal{H}_{c,A}$ is bounded in $L^p(\mathbb{R}^d)$ and
its norm does not exceed $N_p(c,A)$.

(ii)  The condition $N_p(c,A)<\infty$  for the $L^p$ boundedness of $\mathcal{H}_{c,A}$ can't be weakened in general.
\end{lemma}

\begin{proof}
(i) This readily follows from the Minkowskii inequality \cite{BM}.

(ii) Let $1< p< \infty$, $d=1$. Consider the case where $c(k)\ge 0$,  $c(k)=0$ for $k\le 0$, and $A(k)=1/k$ for $k\ge 1$. We proceed as in \cite[Theorem 2.1]{JJS}.

For each $t\in (0,1)$ put
\[
f_t(x)=\frac{1}{x^{1/p}}\chi_{(t,1/t)}(x),\   \  g_t(x)=\frac{1}{x^{1/q}}\chi_{(t,1/t)}(x)\ \ (1/p+1/q=1).
\]
Then
\[
\|f_t\|_p^p=\|g_t\|_q^q=2\log\frac{1}{t}.
\]
For every natural $k$ consider the function
\begin{eqnarray}\label{ht}
h_t(k)&=&\int_0^\infty g_t(kx)f_t(x)dx\\ \nonumber
&=&\int_0^\infty   \frac{1}{(kx)^{1/q}}\chi_{(t,1/t)}(kx)\frac{1}{x^{1/p}}\chi_{(t,1/t)}(x)dx\\ \nonumber
&=&k^{-1/q}\int_0^\infty  \chi_{(t/k,1/kt)}(x)\chi_{(t,1/t)}(x) \frac{dx}{x}.
\end{eqnarray}

If $k>1/t^2$, then  $(t/k,1/kt)\cap (t,1/t)=\varnothing$ and so $\chi_{(t/k,1/kt)}\chi_{(t,1/t)}=0$. In this case, $h_t(k)=0$.

It is easy tho compute then $h_t(1)=2\log\frac{1}{t}$.
Further, if $1<k\le 1/t^2$ we have $(t/k,1/kt)\cap (t,1/t)=(t,1/kt)$.  Therefore  \eqref{ht} implies 
\[
h_t(k)=k^{-1/q}\int_t^{1/kt}\frac{dx}{x}=k^{-1/q}\left(2\log\frac{1}{t}-\log k\right).
\]
Consider the non-negative quantity 
\begin{eqnarray*}
J_t&=&\int_0^\infty g_t(y)(\mathcal{H}_{c,A}f_t)(y)dy\\
&=&\sum_{k=1}^\infty c(k) \int_0^\infty g_t(y)f_t(\frac yk)dy\\
&=&\sum_{k=1}^\infty c(k)k \int_0^\infty g_t(kx)f_t(x)dx\\
&=&\sum_{k=1}^\infty c(k)kh_t(k)\\
&=&c(1)2\log\frac{1}{t}+\sum_{k=2}^\infty c(k)k^{1-1/q}\left(2\log\frac{1}{t}-\log k\right).
\end{eqnarray*}
If $\mathcal{H}_{c,A}$ is bounded in $L^p(\mathbb{R})$, we have by the Holder's inequality 
\[
J_t\le \|\mathcal{H}_{c,A}\|\|f_t\|_p\|g_t\|_q= \|\mathcal{H}_{c,A}\|2\log\frac{1}{t}.
\]
It follows, that
\[
c(1)+\sum_{k=2}^\infty c(k)k^{1/p}\left(1-\frac{\log k}{2\log\frac{1}{t}}\right)\le \|\mathcal{H}_{c,A}\|.
\]

Putting here $t=1/n$ ($n\in \mathbb{N}$), $n \to \infty$ we get by the B. Levi's theorem, that
\[
N_p(c,A)=\sum_{k=1}^\infty c(k)k^{1/p}\le \|\mathcal{H}_{c,A}\|,
\]
which completes the proof in the case  $1< p< \infty$. For the cases $p=1$ and $p=\infty$ the statement (ii) is obvious. 
\end{proof}

 Let $U_j\ (j = 1; \dots ; 2^d)$ be some  fixed enumeration of the family of all open hyperoctants in $\mathbb{R}^d.$  For every
pair $(i; j)$ of indices there is a unique $\varepsilon(i; j)\in\{-1,1\}^d$ such that $$
\varepsilon(i; j)U_i :=
\{(\varepsilon(i; j)_1 x_1;\dots ; \varepsilon(i; j)_d x_d) : x\in U_i\} = U_j.
$$
 It is clear that $\varepsilon(i; j)U_j = U_i$ and
$\varepsilon(i; j)U_l \cap U_i=\varnothing$ as $l \ne j.$
We will assume that $A(k)$ form a commuting family. Then
as is well known there are an orthogonal $d\times d$-matrix $C$ and a family of
diagonal non-singular real matrices $A'(k) = \mathrm{diag}(a_1(k); \dots ; a_d(k))$ such that
$A'(k) = C^{-1}A(k)C$ for $k\in \mathbb{Z}.$ Then 
$$
a(k) := (a_1(k); \dots ; a_d(k))
$$ 
is the family
of all eigenvalues (with every eigenvalue repeated according to its
multiplicity) of the matrix $A(k)$. We put
$$
\Omega_{ij} := \{k\in \mathbb{Z}: (\mathrm{sgn}(a_1(k)); \dots ; \mathrm{sgn}(a_d(k))) = \varepsilon(i; j)\}.
 $$

If $N_2(c,A)<\infty$, i.~e., if $(|\det A(k)|^{-1/2}c(k))_{k\in \mathbb{Z}}\in \ell^1(\mathbb{Z})$, we put
 $$
   \varphi_{ij}(s):=\sum_{k\in \Omega_{ij}}c(k)|a(k)|^{-1/2- \imath s}.
   $$
Above we assume that $|a(k)|^{-1/2-\imath s}:=$ $\prod_{l=1}^d |a_l(k)|^{-1/2-\imath s_l}$ where we put $|a_l(k)|^{-1/2-\imath s_l}
:=\exp((-1/2-\imath s_l)\log |a_l(k)|)$.
 It follows that
\begin{eqnarray*}
\varphi_{ij}(s)
=\sum_{k\in \Omega_{ij}}\frac{c(k)}{\sqrt{|\det A(k)|}}e^{-\imath s\cdot \log |a(k)|}
\end{eqnarray*}
(here $\log |a(k)|:=(\log |a_1(k)|,\dots,\log |a_d(k)|)$;  the dot denotes the inner product in $\mathbb{R}^d$).

Evidently,  $\varphi_{ij}=\varphi_{ji}$ and all functions $\varphi_{ij}$ belong to the algebra $C_b(\mathbb{R}^d)$ of bounded and
continuous functions on $\mathbb{R}^d$ if $N_2(c,A)<\infty$.

Note also that the operator $\mathcal{H}_{c,A}$ is normal if  the matrices  $A(k)$ form a commuting family \cite{Forum}.

\begin{definition} Let the matrices $A(k)$ form a commuting family and $N_2(c,A)<\infty$. We define the  \textit{matrix symbol}
of a Hausdorff  operator $\mathcal{H}_{c, A}$  by
$$
\Phi = \left(\varphi_{ij}\right)_{i,j=1}^{2^d}
$$
\end{definition}
Then $\Phi$ is a symmetric element of the matrix algebra $\mathrm{Mat}_{2^d}(C_b(\mathbb{R}^n)).$

The symbol was first introduced in \cite{faa} for the case of positive definite $A(k)$.

It is known \cite[Theorem 1]{Forum} that
\begin{eqnarray}\label{sigma}
\sigma(\mathcal{H}_{c, A})=\{\lambda \in \mathbb{C}:\inf_{s\in \mathbb{R}^d}|\det(\lambda  I_{2^d}-\Phi(s)|=0\},
\end{eqnarray}
 and by  \cite[Corollary 3]{Forum}
$$
\|\mathcal{H}_{K, A}\|=\max\{|\lambda|:\inf_{s\in \mathbb{R}^d}|\det(\lambda  I_{2^d}-\Phi(s)|=0\}=\sup_{s\in \mathbb{R}^d}\|\Phi(s)\|,
$$
where $\|\Phi(s)\|$ stands for the  norm of the operator in $\mathbb{C}^{2^d}$ of multiplication by the matrix $\Phi(s)$ and $ I_{2^d}$ denotes the unit matrix of order $2^d$.

\section{Main results}

Recall that  real numbers $r_1,\dots, r_m$ are called {\it linear independent over $\mathbb{Z}$} if the
equality $\sum_{k=1}^m l_k r_k=0$ where all $l_k\in \mathbb{Z}$ yields $l_k=0$ for all $k$. As  usual we say that an infinite family of real numbers  is linear independent 
over $\mathbb{Z}$ if
each its finite subfamily   is linear independent over $\mathbb{Z}$.

We shall say that a set  of complex numbers is  {\it rotationally invariant} if it is invariant  under all rotations around  the origin.

\begin{theorem}\label{th:1} Let $A(k)\in \mathrm{GL}(d,\mathbb{R})$  ($k\in \mathbb{Z} $) be a commuting family of self-adjoint matrices, and $N_2(c,A)<\infty$.

(i) Let $a(k) = (a_1(k); \dots ; a_d(k))$ be the family
of all eigenvalues (with their multiplicities) of the matrix $A(k)$. If for some $\nu$ the numbers $\log|a_{\nu}(k)|$ ($k\in \mathbb{Z}$) are linear independent over $\mathbb{Z}$, 
the spectrum $\sigma(\mathcal{H}_{c, A})$ of a Hausdorff  operator $\mathcal{H}_{c, A}$ in $L^2(\mathbb{R}^d)$ is rotationally invariant.


(ii) Let the symbol $\Phi$ of $\mathcal{H}_{c,A}$ be a real analytic matrix function on $\mathbb{R}^d$. Then  $\lambda\in \sigma_p(\mathcal{H}_{c,A})$  if and only if all  matrices $\Phi(s)$ ($s\in \mathbb{R}^d$)  have a common eigenvalue $\lambda$.
\end{theorem}

\begin{proof}
(i)
Consider the truncated  operator
\[
\mathcal{H}_{c,A}^{(n)}f(x):=\sum_{k=-n}^n c(k) f(A(k) x)\  \  (n\in \mathbb{N})
\]
in $L^2(\mathbb{R}^d)$. Its matrix symbol is
$$
\Phi^{(n)} = \left(\varphi_{ij}^{(n)}\right)_{i,j=1}^{2^d},
$$
where
\begin{eqnarray}\label{phi_ij}
\varphi_{ij}^{(n)}(s)&=&\sum_{k\in \Omega_{ij},\atop |k|\le n}c(k)|a(k)|^{-1/2- \imath s}\\\nonumber
&=&\sum_{k\in \Omega_{ij},\atop |k|\le n}\frac{c(k)}{\sqrt{|\det A(k)|}}e^{-\imath s\cdot \log |a(k)|}.
\end{eqnarray}

Let, for definiteness, the numbers $\log|a_{1}(k)|$ ($k\in \mathbb{Z}$) are linear independent over $\mathbb{Z}$.
Then the corollary of Kronecker's approximation theorem (see, e.g., 
\cite{LZh})  implies that the set
\[
\{(e^{-\imath s_1\log |a_1(-n)|},\dots,e^{-\imath s_1\log |a_1(n)|}): s_1\in \mathbb{R}\}
\]
is dense in the $(2n+1)$-dimensional torus $\mathbb{T}^{2n+1}$ for each $n\in \mathbb{N}$. Thus
the set
\begin{equation*}
\{\Lambda(s):=(e^{-\imath s\cdot \log |a(-n)|},\dots,e^{-\imath s\cdot \log |a(n)|}): s\in \mathbb{R}^d\}
\end{equation*}
is dense in $\mathbb{T}^{2n+1}$ for each $n\in \mathbb{N}$, as well.

From \cite[Theorem 1]{Forum} it follows that
\begin{eqnarray}\label{sigma}
\sigma(\mathcal{H}_{c, A}^{(n)})=\{\lambda \in \mathbb{C}:\inf_{s\in \mathbb{R}^d}|\det(\lambda I_{2^d}-\Phi^{(n)}(s)|=0\}.
\end{eqnarray}

For each pare $i,j=1,\dots,2^d$ consider the functions
$$
\xi_{ij}^{(n)}(s):=\sum_{k\in \Omega_{ij},\atop |k|\le n}\frac{c(k)}{\sqrt{|\det A(k)|}}t_k\ \ \ (t=(t_k)\in \mathbb{T}^{2n+1}).
$$
This functions are continuous on $\mathbb{T}^{2n+1}$ and $\varphi_{ij}^{(n)}(s)=\xi_{ij}^{(n)}\circ\Lambda(s)$ for $i,j=1,\dots,2^d$. As was mentioned above the range
 of $\Lambda$ is dense in $\mathbb{T}^{2n+1}$. 
 Then by continuity we get
 \begin{eqnarray}\label{T}
\inf_{s\in \mathbb{R}^d}|\det(\lambda I_{2^d}-\Phi^{(n)}(s)|&=&\inf_{s\in \mathbb{R}^d}|\det(\lambda I_{2^d}-(\xi_{ij}^{(n)}\circ\Lambda(s))_{i,j=1}^{2^d}| \nonumber\\
&=&\inf_{t\in \mathbb{T}^{2n+1}}|\det(\lambda I_{2^d}-(\xi_{ij}^{(n)}(t))_{i,j=1}^{2^d}|.
\end{eqnarray}

Further, for every $\zeta\in \mathbb{C}$, $|\zeta|=1$ we have
 \begin{eqnarray*}
|\det(\zeta\lambda I_{2^d}-\Phi^{(n)}(s)|&=&|\det(\lambda I_{2^d}-\overline{\zeta}(\xi_{ij}^{(n)}(t))_{i,j=1}^{2^d}|\\
&=&|\det(\lambda I_{2^d}-(\xi_{ij}^{(n)}(\overline{\zeta} t))_{i,j=1}^{2^d}|.
 \end{eqnarray*}
The last equation, (\ref{sigma}), and (\ref{T}) imply that $\zeta\sigma(\mathcal{H}_{c, A}^{(n)})=\sigma(\mathcal{H}_{c, A}^{(n)})$. Thus, the set $\sigma(\mathcal{H}_{c, A}^{(n)})$
is rotationally invariant.

Let $R_n:=\mathcal{H}_{c, A}-\mathcal{H}_{c, A}^{(n)}$. Since $N_2(c,A)<\infty$, we have
$$
\|R_n\|\le \sum_{|k|> n}\frac{|c(k)|}{\sqrt{|\det A(k)|}}\to 0\ \   \mbox{ as } n\to \infty .
 $$
Since matrices $A(k)$ form a commuting family, $R_n$ commutes with $\mathcal{H}_{c, A}^{(n)}$ (see the proof of Theorem 1 in \cite{Forum}). Now by \cite[Theorem IV.3.6]{Kato} we have
 \begin{eqnarray}\label{distRn}
\mathrm{dist}_H(\sigma(\mathcal{H}_{c, A}),\sigma(\mathcal{H}_{c, A}^{(n)}))\le \|R_n\|,
 \end{eqnarray}
where $\mathrm{dist}_H(X,Y)$ denotes the Hausdorff distance between compact  sets $X, Y\subset \mathbb{C}$.
Recall that
$$
\mathrm{dist}_H(X,Y)=\inf\{\varepsilon:X\subseteq Y_\varepsilon \mbox { and } Y \subseteq X_\varepsilon\},
$$
where $X_\varepsilon :=\cup_{x\in X}\{z\in \mathbb{C}: |z-x|\le \varepsilon\}$. It is easy to verify that $\mathrm{dist}_H$ is invariant with respect to  rotations:
$\mathrm{dist}_H(\zeta X, \zeta Y)=\mathrm{dist}_H(X,  Y)$ for all $\zeta\in \mathbb{C}$, $|\zeta|=1$. Then
\begin{eqnarray}\label{zeta}
\mathrm{dist}_H(\zeta \sigma(\mathcal{H}_{c, A}), \sigma(\mathcal{H}_{c, A}^{(n)}))&=&\mathrm{dist}_H(\sigma(\mathcal{H}_{c, A}),\overline{\zeta}\sigma(\mathcal{H}_{c, A}^{(n)}))\\\nonumber
&=&\mathrm{dist}_H(\sigma(\mathcal{H}_{c, A}),\sigma(\mathcal{H}_{c, A}^{(n)}))\le \|R_n\|\to 0,
\end{eqnarray}
as $n\to\infty$. Thus, both  $\zeta \sigma(\mathcal{H}_{c, A})$ and $\sigma(\mathcal{H}_{c, A})$ are limits of $\sigma(\mathcal{H}_{c, A}^{(n)}))$ with respect to $\mathrm{dist}_H$.
   It is known that $\mathrm{dist}_H$ is a metric on compact  subsets (see, e.g., \cite[\S 21, VII]{Kuratowski}). Thus, (\ref{zeta}) shows that
    $\zeta \sigma(\mathcal{H}_{c, A})= \sigma(\mathcal{H}_{c, A})$ and
 (i) follows.

(ii) In \cite[Theorem 1]{Forum} it is proven that  the point spectrum $\sigma_p(\mathcal{H}_{c, A})$  consists of such complex numbers $\lambda$ for which the closed set 
$$E(\lambda):=\{s\in \mathbb{R}^d: \det(\lambda-\Phi(s))=0\}
$$
 is of  positive Lebesgue measure.
 
 Thus, if $\lambda\in \cap_{s\in \mathbb{R}^d}\sigma(\Phi(s))$, then  $E(\lambda)=\mathbb{R}^d$  and therefore $\lambda\in\sigma_p(\mathcal{H}_{c, A})$.

Conversely, if $\lambda$ is an eigenvalue of $\mathcal{H}_{c, A}$,  then $\mathrm{mes}(E(\lambda))>0$.
Since each entry $\varphi_{ij}$ of $\Phi$ is a real analytic function on $\mathbb{R}^d$, the function $s\mapsto \det(\lambda-\Phi(s))$ is real analytic, as well. As far as this function vanishes on the set $E(\lambda)$ of positive measure,  it is identically zero
by a version of a uniqueness theorem for real analytic functions \cite{VI} (see also \cite[p. 83]{KP}) and therefore $\lambda$ is a common eigenvalue of $\Phi(s)$ for all $s\in \mathbb{R}^d$.

\end{proof}

\begin{remark}
The arithmetic condition in  Theorem \ref{th:1} (i) is essential. Indeed, consider the operator $\mathcal{H}f(x)=f(x)+f(2x)$ in $L^2(\mathbb{R})$. 
Its scalar symbol (see (\ref{sym}) below) is $\varphi(s)=1+2^{-1/2}e^{-\imath s\log 2}$. Then by \cite{faa} or \cite[Corollary 7]{Forum} we have $\sigma(\mathcal{H})=\mathrm{cl}(\varphi(\mathbb{R}))=\{|z-1|=2^{-1/2}\}$ (here and below $\mathrm{cl}$ stands for the closure).
\end{remark}

\begin{remark}
 Since  $\mathcal{H}_{c,A}$ is normal, the residual spectrum of $\mathcal{H}_{c,A}$ is empty.
\end{remark}

\begin{corollary}  Let $X$ be a closed $\mathcal{H}_{c,A}$-invariant subspace of $L^2(\mathbb{R}^d)$. If the assumptions  of   Theorem \ref{th:1} (i) are true and $\mathcal{H}_{c,A}$ is a minimal normal extension of the restriction $S:=\mathcal{H}_{c,A}|X$, then the spectrum $\sigma(S)$ is rotationally invariant.
\end{corollary}

\begin{proof} In is known (see, e.~g., \cite[Theorem II.2.11]{Conway}) that $\sigma(\mathcal{H}_{c,A})\subseteq \sigma(S)$, and if $\sigma(\mathcal{H}_{c,A})\ne \sigma(S)$
then $\sigma(S)$ is a union of $\sigma(\mathcal{H}_{c,A})$ and some bounded holes 
of $\sigma(\mathcal{H}_{c,A})$ (i.~e. bounded components of $\mathbb{C}\setminus\sigma(\mathcal{H}_{c,A})$). Since the set $\sigma(\mathcal{H}_{c,A})$ is rotationally invariant, each its hole is rotationally invariant, too. This completes the proof.
\end{proof}

\begin{corollary}  Let the assumptions  of  Theorem \ref{th:1} (i) hold. If the operator $\mathcal{H}_{c,A}$ is non-null, then it is not self-adjoint. 
\end{corollary}

\begin{proof} Indeed, since  $\mathcal{H}_{c,A}$ is normal, its spectral radius is  $\|\mathcal{H}_{c,A}\|\ne 0$. Thus, the spectrum $\sigma(\mathcal{H}_{c,A})$ cannot be a subset of reals.
\end{proof}

Recall that the {\it essential Weyl spectrum} $\sigma_{ew}(T)$ of an (closed densely defined) operator $T$ in a complex Banach space $X$ can be defined as 
$$
\sigma_{ew}(T)=\mathbb{C}\setminus \Delta_4(T),
$$
where
$$
\Delta_4(T)=\{\lambda\in\mathbb{C}: T-\lambda I \mbox{ is Fredholm and } \mathrm{ind}(T-\lambda I)=0\}.
$$
It is known that
$$
\sigma_{ew}(T)=\cap_{K\in \mathcal{K}(X)}\sigma(T+K),
$$
where $\mathcal{K}(X)$ stands for the space of compact operators in $X$ (see, e.g., \cite{Gus}, \cite{Berb}, or \cite[Theorem IX.1.4]{EE}.\footnote{In  \cite{EE} the  Weyl spectrum is denoted by  $\sigma_{e4}(T)$.})

In the following $\pi_{00}(T)$ stands for the set of  isolated points of the spectrum $\sigma(T)$
of an operator $T$  that are eigenvalues of finite geometric  multiplicity. 

\begin{corollary}\label{Weyl}   Let $\mathcal{H}_{c,A}\ne O$ and the assumptions  of  Theorem \ref{th:1} (i) hold.
Then  $\sigma_{ew}(\mathcal{H}_{c,A})=\sigma(\mathcal{H}_{c,A})$ if  $0\notin \pi_{00}(\mathcal{H}_{c,A})$, and  $\sigma_{ew}(\mathcal{H}_{c,A})=\sigma(\mathcal{H}_{c,A})\setminus\{0\}$ otherwise. In particular,  $\sigma_{ew}(\mathcal{H}_{c,A})$ is rotationally invariant.
\end{corollary}

\begin{proof} Theorem \ref{th:1} (i) implies that  $\pi_{00}(\mathcal{H}_{c,A})\subseteq \{0\}$. Therefore, the first assertion of the corollary follows from the Weyl theorem, which states that 
\begin{equation}\label{p00}
\sigma_{ew}(\mathcal{H}_{c,A})=\sigma(\mathcal{H}_{c,A})\setminus \pi_{00}(\mathcal{H}_{c,A})
\end{equation}
 (see, e.g., \cite{Gus}, \cite{Berb}). Now the last statement follows from the Theorem \ref{th:1} (i), as well.
\end{proof}

For the next corollaries consider  the \textit{scalar symbol} of a  Hausdorff operator $\mathcal{H}_{c,A}$
\begin{equation}\label{sym}
 \varphi(s):=\sum_{k=-\infty}^\infty c(k)|\det A(k)|^{-1/2}e^{-\imath s\cdot \log |a(k)|}.
\end{equation}

\begin{corollary}\label{Pos} (cf.  \cite{JMS})  Let in addition to the assumptions  of Theorem \ref{th:1} (i) all matrices $(A(k))_{k\in \mathbb{Z}}$ are positive definite and $\mathcal{H}_{c,A}\ne O$. Then the spectrum $\sigma(\mathcal{H}_{c,A})$   is
  an annulus (or a disc) of the form 
  \begin{equation}\label{sp:pos}
 \left\{\zeta\in \mathbb{C}: \inf_{\mathbb{R}^d}|\varphi|\le|\zeta|\le \sup_{\mathbb{R}^d}|\varphi|\right\}
  \end{equation}
and $\|\mathcal{H}_{c,A})\|=\sup_{\mathbb{R}^d}|\varphi|$.  Moreover, $\sigma_{ew}(\mathcal{H}_{c,A})=\sigma(\mathcal{H}_{c,A})$. In particular, $\sigma(\mathcal{H}_{c,A})$ is invariant under compact perturbations of $\mathcal{H}_{c,A}$.
\end{corollary}

\begin{proof}
In the case of positive definiteness the spectrum   $\sigma(\mathcal{H}_{c,A})$ equals to the closure  $\mathrm{cl}(\varphi(\mathbb{R}^d))$  by \cite{faa}, \cite[Corollary 7]{Forum}. Since $N_2(c,A)<\infty$, the scalar symbol $\varphi$ is continuous on $\mathbb{R}^d$ and therefore its range is connected.
Since  the set  $\sigma(\mathcal{H}_{c,A})$ is  connected and rotationally invariant, it is an annulus (or a disc) centered at the origin.  
Moreover, the spectral radius of the operator $\mathcal{H}_{c,A}$ equals to $\sup_{\mathbb{R}^d}|\varphi|$, and equals to  its norm, since $\mathcal{H}_{c,A}$ is normal. This proves the first statement.
 Finally, in our case $\pi_{00}(\mathcal{H}_{c,A})=\varnothing$
 and  the last assertion follows from \eqref{p00}.
\end{proof}

\begin{corollary}\label{Neg} (cf. \cite{JMS}) Let in addition to the assumptions  of Theorem \ref{th:1} (i) all matrices $(A(k))_{k\in \mathbb{Z}}$ are negative definite and $\mathcal{H}_{c,A}\ne O$. Then   the spectrum $\sigma(\mathcal{H}_{c,A})$   is the
annulus (or a disc) of the form \eqref{sp:pos}. Moreover, $\sigma_{ew}(\mathcal{H}_{c,A})=\sigma(\mathcal{H}_{c,A})$. In particular, $\sigma(\mathcal{H}_{c,A})$ is invariant under compact perturbations of $\mathcal{H}_{c,A}$.
\end{corollary}

\begin{proof} The scalar symbol of the operator $\mathcal{H}_{c,A}$ is
\begin{eqnarray*}
 \varphi(s)&:=&\sum_{k=-\infty}^\infty \frac{c(k)}{\sqrt{|\det A(k)|}}e^{-\imath s\cdot \log|a(k)|}\\
 &=&\sum_{k=-\infty}^\infty \frac{c(k)}{\sqrt{\det (-A(k))}}e^{-\imath s\cdot \log(- a(k))}.
\end{eqnarray*}
Thus, $ \varphi$ coincides with  the scalar symbol  $ \varphi^-$ of the operator  $\mathcal{H}_{c,(-A)}$ where all the matrices $(-A(k))$ are positive definite.

According to
\cite[Corollary 8]{Forum},  $\sigma(\mathcal{H}_{c,A})$ equals to the set
$$
-\mathrm{cl}(\varphi(\mathbb{R}^d))\cup \mathrm{cl}(\varphi(\mathbb{R}^d))=-\mathrm{cl}(\varphi^-(\mathbb{R}^d))\cup \mathrm{cl}(\varphi^-(\mathbb{R}^d)).
$$
 But as was shown in the previous corollary $\mathrm{cl}(\varphi^-(\mathbb{R}^d))$ is
 an annulus  (or a disc) centered at the origin and so 
 $-\mathrm{cl}(\varphi^-(\mathbb{R}^d))= \mathrm{cl}(\varphi^-(\mathbb{R}^d))= \mathrm{cl}(\varphi(\mathbb{R}^d))$. It follows that the spectrum
 $\sigma(\mathcal{H}_{c,A})$ is given by the formula \eqref{sp:pos}. The last assertion follows from \eqref{p00} as in the previous corollary.
\end{proof}

Now we shall consider the case, where $A(k)=\mathrm{diag}(a(k),\dots,a(k))$, $a(k)\ne 0$. In other words, we consider discrete Hausdorff operators of the form
\begin{equation}\label{Hca}
\mathcal{H}_{c,a}f(x):=\sum_{k=-\infty}^\infty c(k) f(a(k) x),\ x\in \mathbb{R}^d
\end{equation}
 provided the series converges absolutely.

Evidently every  one-dimensional  discrete Hausdorff operator has the form \eqref{Hca}.
 
 As above, we introduce the scalar symbol of $\mathcal{H}_{c,a}$ as
\begin{equation}
 \varphi(s):=\sum_{k=-\infty}^\infty c(k)|a(k)|^{-d/2}e^{-\imath \log |a(k)|\sum_{j=1}^d s_j}.
\end{equation}

We introduce also the conjugate  scalar symbol of $\mathcal{H}_{c,a}$ as
\begin{equation}
 \varphi^*(s):=\sum_{k=-\infty}^\infty c(k)\mathrm{sgn}(a(k))|a(k)|^{-d/2}e^{-\imath \log |a(k)|\sum_{j=1}^d s_j}.
\end{equation}


\begin{theorem}\label{th:2} Let $N_2(c,a)<\infty$. Then the following assertions hold.

(i) $$
\sigma(\mathcal{H}_{c, a})= \mathrm{cl}(\varphi(\mathbb{R}^d)\cup \varphi^*(\mathbb{R}^d))
$$
($\mathrm{cl}$ stands for the closure), and   $\|\mathcal{H}_{c, a}\|=\max\{\sup|\varphi|, \sup|\varphi^*|\}.$

(ii) Let the set $\{\log|a(k)|:k\in \mathbb{Z}\}$ be linear independent over $ \mathbb{Z}$.  Then the spectrum $\sigma(\mathcal{H}_{c,a})$   is
 an annulus of the form $\{r(c,A)\le|\zeta|\le \|\mathcal{H}_{c,a})\|\}$,  
    or the disc $\{|\zeta|\le \|\mathcal{H}_{c,a})\|\}$. Moreover, $\sigma_{ew}(\mathcal{H}_{c,a})=\sigma(\mathcal{H}_{c,a})$. Thus, $\sigma(\mathcal{H}_{c,a})$ is invariant under compact perturbations of $\mathcal{H}_{c,a}$.
\end{theorem}

\begin{proof} (i) 
In order to employ  Theorem 1 from \cite{Forum} we   enumerate $d$-hyperoctants $U_j$ in such a way that $U_{2^{d-1}+j}=- U_j$ for $j= 1, \dots, 2^{d-1}.$ 
Then $\Omega_{ii}=\{k\in \mathbb{Z}: a(k)>0\},$ $\Omega_{ij}=\{k\in \mathbb{Z}: a(k)<0\}$ if $|j-i|=2^{d-1},$
and $\Omega_{ij}=\varnothing$  otherwise. It follows that
$$
\varphi_{ii}(s)=\varphi_+(s):=\sum_{k: a(k)>0} c(k)|a(k)|^{-d/2}e^{-\imath \log |a(k)|\sum_{j=1}^d s_j}.
$$
Analogously, if $|j-i|=2^{d-1},$
$$
\varphi_{ij}(s)=\varphi_-(s):=\sum_{k: a(k)<0} c(k)|a(k)|^{-d/2}e^{-\imath \log |a(k)|\sum_{j=1}^d s_j},
$$
and $\varphi_{ij}=0$ otherwise. Thus, the matrix symbol of $\mathcal{H}_{c,a}$  is the following block matrix:
$$\Phi=
\begin{pmatrix}
\varphi_+I_{2^{d-1}}&\varphi_- I_{2^{d-1}}\\ \varphi_- I_{2^{d-1}}&\varphi_+I_{2^{d-1}}
\end{pmatrix},
$$
where $I_{2^{d-1}}$ denotes the identity matrix of order
$2^{d-1}.$ Then for every $\lambda\in \mathbb{C}$
$$\lambda -\Phi=
\begin{pmatrix}
(\lambda-\varphi_+)I_{2^{d-1}}&-\varphi_- I_{2^{d-1}}\\ -\varphi_- I_{2^{d-1}}&(\lambda-\varphi_+)I_{2^{d-1}}
\end{pmatrix}
$$
and therefore by the formula of Schur (see, e.~g., \cite[p. 46]{G}),
\begin{eqnarray*}
\det(\lambda-\Phi) &= \det((\lambda-\varphi_+)^2  I_{2^{d-1}}-\varphi_-^2 I_{2^{d-1}})\\ &= ((\lambda-\varphi_+ -\varphi_-)(\lambda-\varphi_++\varphi_-))^{2^{d-1}}\\
&=
((\lambda-\varphi)(\lambda-\varphi^*))^{2^{d-1}},
\end{eqnarray*}
since $\varphi=\varphi_+ +\varphi_-,$   $\varphi^*=\varphi_+ -\varphi_-.$  Now Theorem 1 from \cite{Forum} implies that (we use the boundedness of $\varphi, \varphi^*$) that
$$
\sigma(\mathcal{H}_{c, a})=\{\lambda\in \mathbb{C}: \inf_{s\in \mathbb{R}^d}|(\lambda-\varphi(s))(\lambda-\varphi^*(s))|=0\}= \mathrm{cl}(\varphi(\mathbb{R}^d)\cup \varphi^*(\mathbb{R}^d)).
$$
In view of the normality of $\mathcal{H}_{c, a},$ this yields  that $\|\mathcal{H}_{c, a}\|=\max\{\sup|\varphi|, \sup|\varphi^*|\}.$

(ii)  Consider the truncated  operator
\[
\mathcal{H}_{c,a}^{(n)}f(x):=\sum_{k=-n}^n c(k) f(a(k) x),\  \  n\in \mathbb{N},
\]
in $L^2(\mathbb{R}^d)$.
We claim that 
\begin{equation}\label{sigma:trunk}
\mathrm{cl}(\varphi_n(\mathbb{R}^d))=\left\{\xi_n(t): t=(t_{-n},\dots,t_{n})\in \mathbb{T}^{2n+1} \right\},
\end{equation}
where
\[
\xi_n(t):=\sum_{k=-n}^n c(k)|a(k)|^{-d/2}t_k.
\]
Indeed, the scalar symbol (\ref{sym}) of $\mathcal{H}_{c,a}^{(n)}$ is a trigonometric polynomial of the form
\begin{equation}\label{symbol:n}
\varphi_n(s)=\sum_{k=-n}^n c(k)|a(k)|^{-d/2}e^{-\imath  \log|a(k)|\sum_{j=1}^d s_j},\  s\in \mathbb{R}^d.
\end{equation}

As in the proof of Theorem \ref{th:1} the corollary of 
Kronecker's approximation theorem  implies that the set
\begin{equation}\label{dence}
\{\Lambda(s):=(e^{-\imath \log|a(-n)|\sum_{j=1}^d s_j},\dots, e^{-\imath \log |a(n)|\sum_{j=1}^d s_j}): s\in \mathbb{R}^d\}
\end{equation}
is dense in $\mathbb{T}^{2n+1}$. Moreover, since $\xi_n$ is a continuous and closed map on $\mathbb{T}^{2n+1}$,
we have $\xi_n(\mathrm{cl}(M))=\mathrm{cl}(\xi_n(M))$ for all $M\subset \mathbb{T}^{2n+1}$, 
(see, e.g., \cite[Chapter 1, \S 5, Proposition 9]{Bourb}). Therefore,
\[
\mathrm{cl}(\varphi_n(\mathbb{R}^d))=\mathrm{cl}(\xi_n(\Lambda(\mathbb{R}^d)))=\xi_n(\mathrm{cl}(\Lambda(\mathbb{R}^d)))=\xi_n( \mathbb{T}^{2n+1}).
\]
Similarly, if we  let
\[
\xi_n^*(t):=\sum_{k=-n}^n c(k)\mathrm{sgn}(a(k))|a(k)|^{-d/2}t_k,
\]
then 
\[
\mathrm{cl}(\varphi_n^*(\mathbb{R}^d))=\xi_n^*( \mathbb{T}^{2n+1}).
\]
Since 
$$
\xi_n( \mathbb{T}^{2n+1})=\xi_n^*( \mathbb{T}^{2n+1}),
$$
we conclude that  the spectrum
$$ 
\sigma(\mathcal{H}_{c, a}^{(n)})= \mathrm{cl}(\varphi_n(\mathbb{R}^d)\cup \varphi_n^*(\mathbb{R}^d))=\xi_n(\mathbb{T}^{2n+1})
$$
is connected as the continuous image of the connected set.

Consider the operator $R_n:=\mathcal{H}_{c, a}-\mathcal{H}_{c, a}^{(n)}$. 
Since $\|R_n\|\to 0$ ($n\to\infty$)  (see the proof of Theorem \ref{th:1}),
the formula \eqref{distRn} shows that  $\sigma(\mathcal{H}_{c, a})$ is a limit of $\sigma(\mathcal{H}_{c, a}^{(n)})$ in Hausdorff metric. 
Let $K$ denotes a compact plane set that contains $\sigma(\mathcal{H}_{c, a})$ and all  $\sigma(\mathcal{H}_{c, a}^{(n)})$. By \cite[Chapter 4, \S 42, II, Theorem 2]{Kuratowski2}
 $(2^K)_m=2^K$, where $(2^K)_m$ denotes the space of  closed subspaces of $K$ endowed with Hausdorff metric.  Since the sets $\sigma(\mathcal{H}_{c, a}^{(n)})$ are connected, 
 the set $\sigma(\mathcal{H}_{c, a})$ is also connected  by \cite[Chapter 5, \S 46, II, Theorem 14]{Kuratowski2}. Application of Theorem \ref{th:1} completes the proof of the first assertion of the part (ii). The proof of the second assertion is  the same as for the similar assertion in the Corollary \ref{Pos}.
 \end{proof}

\begin{remark}
Under the conditions of Theorem \ref{th:2} (ii), or corollaries \ref{Pos}, \ref{Neg} the spectrum  $\sigma(\mathcal{H}_{c, a})$ is connected. In general the problem of connectedness of $\sigma(\mathcal{H}_{c, a})$ under the conditions of Theorem \ref{th:1} (i) is open.

\end{remark}

\begin{corollary}\label{Pos:p} (cf.  \cite{JMS}).  Let $A(k)\in \mathrm{GL}(d,\mathbb{R})$  ($k\in \mathbb{Z} $) be a commuting family of positive definite matrices, and $N_2(c,A)<\infty$.  Let the scalar symbol $\varphi$ of   $\mathcal{H}_{c,A}$ be real analytic. Then $\lambda\in \sigma_p(\mathcal{H}_{c,A})$  if and only if $\mathcal{H}_{c,A}=\lambda I$.
\end{corollary}

\begin{proof} One  can assume that $\mathcal{H}_{c,A}\ne O$.  We use the statement (ii). It was shown in \cite[Corollary 7]{Forum} that in the case of positive definiteness one has $\Phi=\mathrm{diag}(\varphi,\dots,\varphi)$.
 It follows that $\lambda\in \sigma(\Phi(s))$ for all $s\in \mathbb{R}^d$, i.~e.,
 $\det(\lambda-\Phi(s))\equiv 0$, if and only if $\varphi(s)\equiv\lambda$. This yields  by \cite{faa}, \cite[Corollary 7]{Forum} that  
 $$
 \sigma(\mathcal{H}_{c,A})= \mathrm{cl}(\varphi(\mathbb{R}^d))=\{\lambda\}.
 $$ 
Since  $\mathcal{H}_{c,A}$ is normal, this implies that $\mathcal{H}_{c,A}=\lambda I$.
 \end{proof}

\begin{corollary}\label{Neg:p}   Let $A(k)\in \mathrm{GL}(d,\mathbb{R})$  ($k\in \mathbb{Z} $) be a commuting family of negative definite matrices,  and $N_2(c,A)<\infty$.  Let the scalar symbol $\varphi$ of $\mathcal{H}_{c,A}$ be real analytic. Then $\lambda\in \sigma_p(\mathcal{H}_{c,A})$  if and only if 
$\mathcal{H}_{c,A}=\pm\lambda J$, where $Jf(x)= f(-x)$.
\end{corollary}

\begin{proof}  One  can assume that $\mathcal{H}_{c,A}\ne O$. First note that if a function $f_e$ from $L^2(\mathbb{R}^d)$ is even that
\[
\mathcal{H}_{c,A}f_e(x)=\sum_{k\in (\mathbb{Z}}c(k)f_e(-A(k)x)=\mathcal{H}_{c,(-A)}f_e(x).
\]
Similarly, if a function $f_o$ from $L^2(\mathbb{R}^d)$ is odd that
\[
\mathcal{H}_{c,A}f_o(x)=-\sum_{k\in (\mathbb{Z}}c(k)f_e(-A(k)x)=-\mathcal{H}_{c,(-A)}f_e(x).
\]
Now let $f$ from $L^2(\mathbb{R}^d)$ is an eigenfunction of $\mathcal{H}_{c,A}$ with an eigenvalue $\lambda$, i.e., $\mathcal{H}_{c,A}f=\lambda f$, and $f\ne 0$.
We have $f=f_e+f_o$ with an even part  $f_e\in L^2(\mathbb{R}^d)$ and an odd part $f_o\in L^2(\mathbb{R}^d)$, and therefore $\mathcal{H}_{c,A}f_e+\mathcal{H}_{c,A}f_o=\lambda f_e+\lambda f_o$.   Due to the previous equalities, this implies    that 
\[
\mathcal{H}_{c,(-A)}f_e-\lambda f_e=\mathcal{H}_{c,(-A)}f_o+\lambda f_o.
\]
Since the left-hand side here is even, and  the right-hand one  is odd, this yields that both functions are zero. Thus,
\[
\mathcal{H}_{c,(-A)}f_e=\lambda f_e, \mbox{ and } \mathcal{H}_{c,(-A)}f_o=(-\lambda) f_o.
\]
Let $f_e\ne 0$. Since $-A(k)$ is positive definite, Corollary \ref{Pos:p} implies  that $\mathcal{H}_{c,(-A)}=\lambda I$, i.e., $\mathcal{H}_{c,A}=\lambda J$.
Similarly, if $f_o\ne 0$, we have by the Corollary \ref{Pos:p}, that $\mathcal{H}_{c,A}=-\lambda J$.
The converse is obvious. 
\end{proof}

\section{Several examples}

Below we list several examples of functional differential equations with discrete Hausdorff operators.

 \begin{example} The functional  differential   equation
\begin{equation*}
y'(t)=\sum_{k\in E} c(k)y(a(k)t)
\end{equation*}
($E$ is a finite subset of $\mathbb{Z}_+$) with a one-dimensional discrete Hausdorff operator $\mathcal{H}_{c,a}$ of the form \eqref{Hca} in the right-hand side is called the multi-pantograph equation.
Theorem \ref{th:2} describes the spectrum of the operator  $\mathcal{H}_{c,a}$ even for the case where $E=\mathbb{Z}$.  
\end{example}

Special cases  and analogs of  the multi-pantograph equation  have found applications in number theory,  
dynamical systems, probability, quantum mechanics, a current
collection system for an electric locomotive, biology, economy, control,
and electrodynamics (see, e.g., \cite{BellenZennaro}, \cite{Liu-Li}).

  \begin{example}
{\rm The next functional  partial differential equation  arises in a model of cell growth
\begin{equation*}
\frac{\partial n(x,t)}{\partial  t}+q\frac{\partial n(x,t)}{\partial  x}=c(0)n(x,t)+c(1)n(\alpha x,t)\quad (q>0, \alpha>0)
\end{equation*}
(see, e.g., \cite{Zaidi-Van Brunt-Wake}). The right-hand side here is a two-term two-dimensional discrete Hausdorff operator with positive definite commuting perturbation matrices
$A(0)=I_2$, $A(1)=\mathrm{diag}(\alpha,1)$. Formula \eqref{sigma} implies that the spectrum of this operator is the circle with center $c(0)$ and radius $|c(1)|/\sqrt{\alpha}$.}
 \end{example}

 \begin{example}
{\rm In \cite{Ross} the following functional  partial differential equation 
\begin{equation}\label{DeltaR}
-\Delta Ru(x)=f(x)
\end{equation}
  was considered   with  a discrete Hausdorff operator
\[
Ru(x)=\sum_{k\in E}c(k)u(q^{-k}x),
\]
 where $c(k)\in\mathbb{C}$, $q>1$, $E$ is a finite subset of $\mathbb{Z}$. In this case, $A(k)=\mathrm{diag}(q^{-k},\dots,q^{-k})$. Here  Theorem \ref{th:2} is applicable. 
 One has
 \[
 \varphi(s)=\varphi^*(s)=
\sum_{k\in E}c(k)(q^{\frac{d}{2}}e^{-\imath (\log q)\sum_{j=1}^d s_j})^k,\  s\in \mathbb{R}^d.
 \]
 Thus, by the aforementioned theorem we have $\sigma(R)=r(q^{\frac{d}{2}}\mathbb{T})$ in $L^2(\mathbb{R}^d)$, where
 \[
 r(z)=\sum_{k\in E}c(k)z^k.
 \]
This  result was obtained in \cite[\S 1.1]{Ross} using a different method. In particular, $R$ is invertible if and only if $r(z)\ne 0$ for $|z|=q^{\frac{d}{2}}$. In this case, equation \eqref{DeltaR} in $L^2(\mathbb{R}^d)$ reduces to the Poisson equation. Since  $R$ is normal,  it follows that $\|R\|=\max\{|r(z)|: |z|=q^{\frac{d}{2}}\}$. Under the condition of boundedness given by Lemma \ref{lem1}
these results are valid for the case of infinite $E\subseteq\mathbb{Z}$, too.}
 \end{example} 

We shall give some applications of Theorem \ref{th:1} to several pantograph-type partial differential equations.
 
  \begin{example}
{\rm Consider the following multidimensional pantograph-type PDE in $L^2(\mathbb{R}^d)$
$$
\frac{\partial u(t,\cdot)}{\partial t}=\sum_{k=-\infty}^\infty c(k)u(t, A(k)\cdot)+Ku(t, \cdot)\equiv(\mathcal{H}_{c,A}+K)u(t,\cdot)
$$
with unknown differentiable function   $t\mapsto u(t,\cdot)$ $:\mathbb{R}\to L^2(\mathbb{R}^d)$ where  $\mathcal{H}_{c,A}$ and a compact operator $K$  in $L^2(\mathbb{R}^d)$ act with respect to $x$.
As usual we can rewrite the last equation in the form
\begin{equation}\label{partpantograph3}
\frac{du}{dt}=(\mathcal{H}_{c,A}+K)u.
\end{equation}
 It is known \cite[\S II.3, Theorem 3.1]{Daleckii-Krein} that if all solutions of 
 such equation are  bounded on $\mathbb{R}$, then $\sigma(\mathcal{H}_{c,A}+K)\subset \imath\mathbb{R}$.
Let    
$\mathcal{H}_{c,A}\ne O$. Since $\mathcal{H}_{c,A}$ is normal, it follows  that $\sigma(\mathcal{H}_{c,A})\ne \{0\}$. Then by Corollary \ref{Weyl} under the  conditions of  Theorem \ref{th:1}(i) we have $\sigma_{ew}({\mathcal{H}}_{c,A})\ne  \{0\}$. In turn, this implies that  $\sigma_{ew}({\mathcal{H}}_{c,A}+K)= \sigma_{ew}({\mathcal{H}}_{c,A})\ne \{0\}$, too. Since $\sigma_{ew}({\mathcal{H}}_{c,A})$ is rotationally invariant, we conclude that 
 the equation (\ref{partpantograph3}) has unbounded solutions  in $L^2(\mathbb{R}^d)$.}
 \end{example}

 \begin{example}
{\rm Similarly to  the previous example, Theorem 3.2 from  \cite[\S II.3]{Daleckii-Krein} shows that under the  conditions of  Theorem \ref{th:1}(i) 
if $\mathcal{H}_{c,A}\ne O$   and $K$ is a compact  operator   in $L^2(\mathbb{R}^d)$ the  equation 
$$
\frac{d^2u}{dt^2}=(\mathcal{H}_{c,A}+K)u
$$
 has unbounded solutions in $L^2(\mathbb{R}^d)$    by Corollary \ref{Weyl}.}
 \end{example}

  \begin{example} 
{\rm  Consider an inhomogeneous equation
\begin{equation}\label{inhom}
\frac{du}{dt}=(\mathcal{H}_{c,A}+K)u+f
\end{equation}
with given continuous function  $f:\mathbb{R}\to L^2(\mathbb{R}^d)$ (again, $\mathcal{H}_{c,A}$ and compact operator $K$ act in $L^2(\mathbb{R}^d)$). We say that this equation has the bounded uniqueness  property if for each bounded $f$ it has a unique bounded solution $u:\mathbb{R}\to L^2(\mathbb{R}^d)$. It is known \cite[\S II.4, Theorem 4.1]{Daleckii-Krein} that if $\sigma(\mathcal{H}_{c,A}+K)$ is disconnected the equation (\ref{inhom}) has the bounded uniqueness  property if and only if the spectrum $\sigma(\mathcal{H}_{c,A}+K)$ does not intersect $\imath\mathbb{R}$. Thus, under the  conditions of  Theorem \ref{th:1}(i) if  
$\mathcal{H}_{c,A}\ne O$ Corollary \ref{Weyl} shows that
the  equation (\ref{inhom}) does not enjoy the bounded uniqueness property if $\sigma(\mathcal{H}_{c,A})$ is not an annulus  (or a disc).}
 \end{example}

\section{acknowledgments}
The author is partially supported by the State Program of Scientific Research
of Republic of Belarus, project No. 20211776
 and by the Ministry of Education and Science of Russia,  agreement No. 075-02-2023-924.

\section{Data availability statement}
The author confirms that all datal generated or analyzed during this study
are included in this article.
This work does not have any conflicts of interest.

\

Department of Mathematics and Programming Technologies, Francisk Skorina Gomel State University, Gomel, 246019, Belarus $\&$ Regional Mathematical Center, Southern Federal
University, Rostov-on-Don, 344090, Russia.

\end{document}